\documentclass[11pt]{amsart}
\usepackage{amsthm, amsfonts, amssymb, amsmath, amscd, url}
\usepackage[all]{xy}
\usepackage{lmodern}
\usepackage[margin=3.5cm]{geometry}

\newcommand{\Q}{\mathbb{Q}}
\newcommand{\Z}{\mathbb{Z}}

\newcommand{\G}{\Gamma}

\newcommand{\la}{\langle}
\newcommand{\ra}{\rangle}
\newcommand{\SL}{\mathrm{SL}}
\newcommand{\GL}{\mathrm{GL}}
\newcommand{\M}{\mathrm{M}}

\newtheorem{thm}{Theorem}[section]
\newtheorem{lem}[thm]{Lemma}

\newtheorem{prop}[thm]{Proposition}

\newtheorem*{thm*}{Theorem}

\theoremstyle{definition}

\newtheorem{rem}[thm]{Remark}

\newtheorem*{convention}{Convention}
\newtheorem*{comments}{Comments}
\newtheorem*{content}{Contents}

\begin{document}
\title{Linear groups - Malcev's theorem and Selberg's lemma}
\author{Bogdan Nica}
\address{Mathematisches Institut, Georg-August Universit\"at G\"ottingen}
\email{bogdan.nica@gmail.com}
\date{June 2013.}
\begin{abstract} These notes give an account of two fundamental facts concerning finitely generated linear groups: Malcev's theorem on residual finiteness, and Selberg's lemma on virtual torsion-freeness.
\end{abstract}
\maketitle

%%%%%%%%%%%%%%%%%%%%%%
\section*{Introduction}
A group is \textbf{linear} if it is (isomorphic to) a subgroup of $\GL_n(K)$, where $K$ is a field. If we want to specify the field, we say that the group is linear over $K$. The following  theorems are fundamental, at least from the perspective of combinatorial group theory.

\begin{thm*}[Malcev 1940] A finitely generated linear group is residually finite.
\end{thm*}

\begin{thm*}[Selberg 1960] A finitely generated linear group over a field of zero characteristic is virtually torsion-free.
\end{thm*}

A group is \textbf{residually finite} if its elements are distinguished by the finite quotients of the group, i.e., if each non-trivial element of the group remains non-trivial in a finite quotient. A group is \textbf{virtually torsion-free} if some finite-index subgroup is torsion-free. As a matter of further terminology, Selberg's theorem is usually referred to as Selberg's lemma, and Malcev is alternatively transliterated as Mal'cev or Maltsev.

\begin{content} The main body of this text has three sections. In the first one we discuss residual finiteness and virtual torsion-freeness, with emphasis on their relation to a third property - roughly speaking, a ``$p$-adic'' refinement of residual finiteness. The main theorem we are actually interested in, due to Platonov (1968), gives such refined residual properties for finitely generated linear groups. Both Malcev's theorem and Selberg's lemma are consequences of this more powerful, but lesser known, theorem of Platonov. The second section is devoted to $\SL_n(\Z)$. This example is too important, too interesting, too much fun to receive anything less than a scenic analysis. In the last section we return to a proof of Platonov's theorem. \end{content}

\begin{comments}
Besides the Russian original \cite{P}, the only other source in the literature for Platonov's theorem that I am aware of is the account of Wehrfritz in \cite[Chapter 4]{W}. The proof presented herein is, I think, considerably simpler. It is mainly influenced by the discussion of Malcev's theorem in the lecture notes by Stallings \cite{S}, and it is quite similar to Platonov's own arguments in \cite{P}. 

An alternative road to Selberg's lemma is to use valuations. This is the approach taken by Cassels in \emph{Local fields} (Cambridge University Press 1986), and by Ratcliffe in \emph{Foundations of hyperbolic manifolds} (2nd edition, Springer 2006).

I thank, in chronological order, Andy Putman for a useful answer via \url{mathoverflow.net}, Jean-Fran\c{c}ois Planchat for a careful reading and constructive comments, and Vadim Alekseev for a translation of Platonov's article.
\end{comments}

\begin{convention} In this text, rings are commutative and unital.
\end{convention}
%%%%%%%%%%%%%%%%%%%%%%
\section{Virtual and residual properties of groups}\label{VR}
Virtual torsion-freeness and residual finiteness are instances of the following terminology. Let $\mathcal{P}$ be a group-theoretic property. A group is \textbf{virtually $\mathcal{P}$} if it has a finite-index subgroup enjoying $\mathcal{P}$. A group is \textbf{residually $\mathcal{P}$} if each non-trivial element of the group remains non-trivial in some quotient group enjoying $\mathcal{P}$. The virtually $\mathcal{P}$ groups and the residually $\mathcal{P}$ groups contain the $\mathcal{P}$ groups. It may certainly happen that a property is virtually stable (e.g., finiteness) or residually stable (e.g., torsion-freeness). 

Besides virtual torsion-freeness and residual finiteness, we are interested in the hybrid notion of \textbf{virtual residual $p$-finiteness} where $p$ is a prime. This is obtained by  residualizing the property of being a finite $p$-group, followed by the virtual extension. The notion of virtual residual $p$-finiteness has, in fact, the leading role in this account for it relates both to residual finiteness and to virtual torsion-freeness.

\begin{lem}[``Going down'']
If $\mathcal{P}$ is inherited by subgroups, then both virtually $\mathcal{P}$ and residually $\mathcal{P}$ are inherited by subgroups. In particular, virtual torsion-freeness, residual finiteness, and virtual residual $p$-finiteness are inherited by subgroups.
\end{lem}

\begin{lem}[``Going up'']\label{going up}
Virtually $\mathcal{P}$ passes to finite-index supergroups. In particular, both virtual torsion-freeness and virtual residual $p$-finiteness pass to finite-index supergroups. Residual finiteness passes to finite-index supergroups.
\end{lem}

Observe that residual $p$-finiteness, just like torsion-freeness, is not virtually stable. Residual finiteness does pass to finite-index supergroups because of the following equivalent description: a group is residually finite if and only if every non-trivial element lies outside of a finite-index subgroup.

Residual $p$-finiteness trivially implies residual finiteness. Going up, we obtain:

\begin{prop}\label{one p suffices}
Virtual residual $p$-finiteness for some prime $p$ implies residual finiteness.
\end{prop}

On the other hand, residual $p$-finiteness imposes torsion restrictions. Namely, in a residually $p$-finite group, the order of a torsion element must be a $p$-th power. Hence, if a group is residually $p$-finite and residually $q$-finite for two different primes $p$ and $q$, then it is torsion-free. Virtualizing this statement, we obtain:

\begin{prop}\label{p and q}
Virtual residual $p$-finiteness and virtual residual $q$-finiteness for two different primes $p$ and $q$ imply virtual torsion-freeness.
\end{prop}

In light of Propositions~\ref{one p suffices} and ~\ref{p and q}, we see that Malcev's theorem and Selberg's lemma are consequences of the following:

\begin{thm*}[Platonov 1968] Let $G$ be a finitely generated linear group over a field $K$. If $\mathrm{char}\:K=0$, then $G$ is virtually residually $p$-finite for all but finitely many primes $p$. If $\mathrm{char}\:K=p$, then $G$ is virtually residually $p$-finite.
\end{thm*}

%%%%%%%%%%%%%%%%%%%%%%
\section{Essellennzee}\label{SL}
In this section we examine $\SL_n(\Z)$, where $n\geq 2$. We start with a most familiar fact.

\begin{prop}\label{SLn f.g.}
$\SL_n(\Z)$ is generated by the elementary matrices $\{1_n+e_{ij}: i\neq j\}$.
\end{prop}

\begin{proof}
In general, if $A$ is a euclidean domain then $\SL_n(A)$ is generated by the elementary matrices $\{1_n+a\cdot e_{ij}:  a\in A, i\neq j\}$. The first step is to turn any matrix in $\SL_n(A)$ into a diagonal one via elementary operations. Using division with remainder, and the euclidean map on $A$ as a way of measuring the decrease in complexity, we can insure that a single non-zero entry, say $a$, remains in the first row. Column swapping brings $a$ in position $(1,1)$, and row reductions using the invertible $a$ turn all other entries in the first column to $0$. Now repeat this procedure for the remaining $(n-1)\times (n-1)$ block. The second step is to bring a diagonal matrix of determinant $1$ to the identity matrix $1_n$ via elementary operations. This is done using the transition
\begin{align*} \begin{pmatrix}
a & 0 \\
0 & b
\end{pmatrix} \leadsto
\begin{pmatrix}
a & a \\
0 & b
\end{pmatrix} \leadsto
\begin{pmatrix}
1 & a \\
(a^{-1}-1)b & b
\end{pmatrix} \leadsto
\begin{pmatrix}
1 & a \\
0 & ab
\end{pmatrix} \leadsto
\begin{pmatrix}
1 & 0 \\
0 & ab
\end{pmatrix}.\end{align*}
Finally, if the additive group of $A$ is generated by $\{a_1,\dots, a_k\}$, then the corresponding matrices $\{1_n+a_1\cdot e_{ij},\dots, 1_n+a_k\cdot e_{ij}: i\neq j\}$ generate all the elementary matrices.
\end{proof}

Let $N$ be a positive integer. Reduction modulo $N$ defines a group homomorphism $\SL_n(\Z)\to \SL_n(\Z/N)$, which enjoys the following remarkable property:

\begin{lem}
The congruence homomorphism $\SL_n(\Z)\to \SL_n(\Z/N)$ is onto.
\end{lem}

\begin{proof} Since $\SL_n(\Z)$ is generated by the elementary matrices, and the elementary matrices are mapped to the elementary matrices by the congruence homomorphism, its surjectivity is equivalent to the elementary generation of $\SL_n(\Z/N)$. The Chinese Remainder Theorem provides a decomposition $\Z/N\simeq \prod \Z/{p_i^{s_i}}$ into local rings, mirroring the decomposition $N=\prod p_i^{s_i}$ into primes. Direct products preserve elementary generation, and we now show that it holds over local rings. This is actually easier than elementary generation for euclidean domains. Let $A$ be a local ring and pick a matrix in $\SL_n(A)$. Some first-row entry is not in $\pi$, the maximal ideal of $A$, so it is invertible in $A$. Column swapping brings this element in the $(1,1)$-position, and then the first row and the first column can be cleared. The rest goes as in the proof of Proposition~\ref{SLn f.g.}. 
\end{proof}

The kernel of the congruence homomorphism
\begin{align*}\G(N):=\ker \big( \SL_n(\Z)\to \SL_n(\Z/N)\big)=\big\{X\in \SL_n(\Z): X \equiv 1_n \;\textrm{mod}\: N\big\}\end{align*}
is the \textbf{principal congruence subgroup of level $N$}. In particular, $\G(1)=\SL_n(\Z)$.

The principal congruence subgroups are normal, finite-index subgroups of $\SL_n(\Z)$. The following lemma provides the formula for their index.

\begin{lem}\label{computing the index} The index of $\G(N)$ is given by
\begin{align*}
[\G(1):\G(N)]=|\SL_n(\Z/N)|=N^{n^2-1}\prod_{p|N}\Big(\prod_{i=2}^{n}(1-p^{-i})\Big).
\end{align*}
\end{lem}

\begin{proof}
First, recall that
\begin{align*}|\SL_n(\Z/p)|=\frac{1}{p-1}|\GL_n(\Z/p)|=\frac{1}{p-1}\prod_{i=0}^{n-1}(p^n-p^i)=p^{n^2-1}\prod_{i=2}^{n}(1-p^{-i}).
\end{align*}
Next, we find the size of $\SL_n(\Z/{p^k})$. Consider again the general case of a local ring $A$ with maximal ideal $\pi$. The congruence map $\GL_n(A)\to \GL_n(A/\pi)$ is onto: \emph{any} lift to $A$ of a matrix in $\GL_n(A/\pi)$ has determinant not in $\pi$, i.e., invertible in $A$. Furthermore, the kernel of the congruence map $\GL_n(A)\to \GL_n(A/\pi)$ is $1_n+\M_n(\pi)$ since a matrix congruent to $1_n$ modulo $\pi$ has determinant in $1+\pi$, hence invertible in $A$. Thus, if $A$ is also finite, then
\begin{align*}|\GL_n(A)|=|\pi|^{n^2}\cdot |\GL_n(A/\pi)|.\end{align*}
Now $|\GL_n|=|\GL_1|\cdot |\SL_n|$ over any ring, and $|\GL_1(A)|=|\pi|\cdot |\GL_1(A/\pi)|$, so
\begin{align*}
|\SL_n(A)|=|\pi|^{n^2-1}\cdot |\SL_n(A/\pi)|.
\end{align*}
Returning to the particular case we are interested in, we obtain
\begin{align*}
|\SL_n(\Z/{p^k})|=p^{(k-1)(n^2-1)} |\SL_n(\Z/{p})|=(p^k)^{n^2-1}\prod_{i=2}^{n}(1-p^{-i}).
\end{align*}
Finally, the size of $\SL_n(\Z/N)$ is obtained by multiplying the above formula over the prime decomposition of $N$.
\end{proof}

Bass - Lazard - Serre (Bull. Amer. Math. Soc. 1964) and Mennicke (Ann. Math. 1965) have shown that, for $n\geq 3$, every finite-index subgroup of $\SL_n(\Z)$ contains some principal congruence subgroup. This Congruence Subgroup Property is definitely not true for $n=2$. The first one to state this failure was Klein (1880), and proofs were subsequently provided by Fricke (1886) and Pick (1886). It is in fact known by now that it is an exceptional feature for a finite-index subgroup of $\SL_2(\Z)$ to contain a principal congruence subgroup.

The principal congruence subgroups are organized according to the divisibility of their levels: $\G(M)\supseteq \G(N) \Leftrightarrow M|N$, that is, ``to contain is to divide''. This puts the emphasis on the prime stratum $\{\G(p) : p \textrm{ prime}\}$, and on the descending chains $\{\G(p^k) : k\geq 1\}$ corresponding to each prime $p$. Observe that 
$\cap_{p}\: \G(p)=\{1_n\}$, and that $\cap_{k}\: \G(p^k)=\{1_n\}$ for each prime $p$, meaning that the elements of $\SL_n(\Z)$ can be distinguished both along the prime stratum, as well as along each descending $p$-chain. Thus:

\begin{prop}
$\SL_n(\Z)$ is residually finite.
\end{prop}

Clearly $\SL_n(\Z)$ is not torsion-free. For example, 
\begin{align*} \begin{pmatrix} 0 & -1\\ 1 & 0\end{pmatrix}, \qquad \begin{pmatrix} 0 & -1\\ 1 & 1\end{pmatrix}\end{align*} 
are elements of order $4$, respectively $6$, in $\SL_2(\Z)$. However, we have:
\begin{prop}
$\SL_n(\Z)$ is virtually torsion-free.
\end{prop}
This is an immediate consequence of the following fact, due to Minkowski (1887):
\begin{lem}
$\G(N)$ is torsion-free provided $N\geq 3$.
\end{lem}

\begin{proof} It suffices to show that $\G (4)$ and $\G(p)$, where $p\geq 3$ is a prime, are torsion-free. Let $p$ be any prime, and assume that $X\in \G(p)$ is a non-trivial element having finite order. Up to replacing $X$ by a power of itself, we may assume that $X^q=1_n$ for a prime $q$. Then
\begin{align*}
-q (X-1_n)=\sum_{i\geq 2}^q \binom{q}{i} (X-1_n)^i.
\end{align*}
Let $p^s$, where $s\geq 1$, be the highest power of $p$ dividing all the entries of $X-1_n$. The left hand side of the displayed identity is divisible by at most $p^s$ if $q\neq p$, and by at most $p^{s+1}$ if $q=p$. The right hand side is divisible by $p^{2s}$, and even by $p^{2s+1}$ if $q=p\geq 3$. Hence $q=p=2$ and $s=1$. The conclusion that $p=2$ and $s=1$ means that $\G(2)$ is the only one in the prime stratum which harbours torsion, and that $\G(4)$, its successor in the descending $2$-chain, is free of torsion. The conclusion that $q=2$ means that torsion elements in $\G(2)$ have order a power of $2$. As $X^2\in \G(4)$ whenever $X\in \G(2)$, and $\G(4)$ is torsion-free, it follows that non-trivial torsion elements in $\G(2)$ have order $2$.
\end{proof}

This lemma can be used to control the torsion spectrum - that is, the possible orders of torsion elements - in $\SL_n(\Z)$. Let us illustrate the basic idea in the simplest case, when $n=2$. Given a group homomorphism, the torsion spectra of its domain, kernel, and range are trivially related by $\tau(\mathrm{dom})\subseteq \tau(\mathrm{ker})\cdot \tau(\mathrm{ran})$. In the case of a congruence homomorphism, this reads as $\tau(\SL_2(\Z))\subseteq \tau(\G(N))\cdot \tau(\SL_2(\Z/N))$. If $N=3$ then $\tau(\G(3))=\{1\}$, and it can be checked that $\tau(\SL_2(\Z/3))=\{1,2,3,4,6\}$. Somewhat easier, in fact, is to let $N=2$: then $\tau(\G(2))=\{1,2\}$, and it is immediate that $\tau(\SL_2(\Z/2))=\{1,2,3\}$. We conclude that $\tau(\SL_2(\Z))\subseteq \{1,2,3,4,6\}$. Equality holds, actually, since there are elements of order $4$ and $6$ in $\SL_2(\Z)$. 

The presence of a torsion element with composite order, namely $6$, implies that $\SL_n(\Z)$ is not residually $p$-finite for any prime $p$. As with torsion-freeness, this is easily remedied by passing to a finite-index subgroup:

\begin{prop}
$\SL_n(\Z)$ is virtually residually $p$-finite for each prime $p$.
\end{prop}

More precisely, we show:

\begin{lem}\label{residual p-finiteness of PCS}
$\G(N)$ is residually $p$-finite for each prime $p$ dividing $N$.
\end{lem}

\begin{proof} It suffices to prove that, for each prime $p$, $\G(p)$ is residually $p$-finite. To that end, we claim that each successive quotient $\G(p^{k})/\G(p^{k+1})$ in the descending chain $\{\G(p^k) : k\geq 1\}$ is a $p$-group. This is seen most directly by observing that each element in $\G(p^{k})/\G(p^{k+1})$ has order $p$: for any matrix $1_n+p^kX\in \G(p^k)$ we have
\begin{align*}(1_n+p^kX)^p=1_n+\sum _{i=1}^p \binom{p}{i} p^{ki}X^i\in \G(p^{k+1}).\end{align*}

Another way is to use the formula for the index (Lemma~\ref{computing the index}), which yields that $\G(p^{k})/\G(p^{k+1})$ has size $p^{n^2-1}$.

A third, more involved argument shows that each successive quotient $\G(p^{k})/\G(p^{k+1})$ is isomorphic to $(\Z/p,+)^{n^2-1}$. Start by noting that, for any matrix $1_n+p^kX\in \G(p^k)$, we have $1=\det (1_n+p^kX)=1+p^k\:\mathrm{tr}(X)\textrm{ mod } p^{2k}$; in particular, $p$ divides $\mathrm{tr} (X)$. Let $\mathfrak{sl}_n(\Z/p)$ denote the additive group of traceless $n\times n$ matrices over $\Z/p$. Then the map
\begin{align*} \phi_k: \G(p^k)\to \mathfrak{sl}_n(\Z/p), \quad \phi_k(1_n+p^kX)= X \textrm{ mod } p\end{align*}
is well-defined. Firstly, $\phi_k$ is a homomorphism: for $1_n+p^kX$ and $1_n+p^kY$ in $\G(p^k)$ we have
$\phi_k\big((1_n+p^kX)(1_n+p^kY)\big)= X+Y+p^kXY \textrm{ mod } p=X+Y \textrm{ mod } p$.
Secondly, the kernel of $\phi_k$ is $\G(p^{k+1})$. Thirdly, we claim that $\phi_k$ is onto. The target group $\mathfrak{sl}_n(\Z/p)$ is generated by the $n^2-n$ off-diagonal matrix units $\{e_{ij}: 1\leq i\neq j \leq n\}$ together with the $n-1$ diagonal differences $\{e_{ii}-e_{i+1\: i+1}: 1\leq i\leq n-1\}$. It is immediate that the off-diagonal matrix units are in the image of $\phi_k$, as $\phi_k(1_n+p^ke_{ij})=e_{ij}$ for $i\neq j$. To obtain the diagonal differences, consider an $n\times n$ matrix having the $2\times 2$-block
\begin{align*}
\begin{pmatrix} 1+p^k& p^k\\-p^k &1-p^k \end{pmatrix}
\end{align*}
on the diagonal, all the other non-zero entries being $1$'s along the remaining diagonal slots. This is a matrix in $\G(p^k)$ which is mapped by $\phi_k$ to $e_{ii}-e_{i+1\: i+1} + e_{i\: i+1}-e_{i+1\: i}$. As $e_{i\: i+1}$ and $e_{i+1\: i}$ are in the image of $\phi_k$, the same is true for $e_{ii}-e_{i+1\: i+1}$. \end{proof}

\begin{rem}\label{what lurks} Scratch most properties of $\SL_n(\Z)$ and you will find a great discrepancy between $\SL_2(\Z)$ and $\SL_{n\geq 3}(\Z)$ lurking underneath. For the discussion at hand, the difference turns out to be the following: $\SL_2(\Z)$ admits finite-index subgroups which are residually $p$-finite for all primes $p$, whereas in $\SL_{n\geq 3}(\Z)$ every finite-index subgroup is residually $p$-finite for only finitely many primes $p$. The question which clarifies and sharpens this contrast is whether principal congruence subgroups can be residually $p$-finite for a prime $p$ not dividing the level.

In $\SL_2(\Z)$, the answer is that $\G(2)$ is residually $p$-finite for $p=2$ only, but $\G(N)$ with $N\geq 3$ is residually $p$-finite for all primes $p$. The exceptional case is due to the $2$-torsion in $\G(2)$. In the higher level case there is no torsion. Now a torsion-free subgroup of $\SL_2(\Z)$ is free, since $\SL_2(\Z)$ acts on a tree with finite vertex stabilizers and without inversion (see I.\S 4 of Serre's \emph{Trees}, Springer 1980). Thus $\G(N)$ with $N\geq 3$ is free. We may then use mutual abstract embeddings to conclude that $\G(N)$ with $N\geq 3$, in fact every free group, is residually $p$-finite for all primes $p$.

In $\SL_{n\geq 3}(\Z)$, the answer is that $\G(N)$ is residually $p$-finite if and only if $p$ divides $N$. Once we know this, the Congruence Subgroup Property will imply that each finite-index subgroup of $\SL_{n\geq 3}(\Z)$ is residually $p$-finite for only finitely many primes $p$. Now let us justify the answer, specifically the forward implication. The proof hinges on computing the abelianization of $\G (N)$, and this is essentially due to Lee and Szczarba (Invent. Math. 1976). As in the proof of Lemma~\ref{residual p-finiteness of PCS}, there is a well-defined homomorphism 
\begin{align*} \G(N)\to \mathfrak{sl}_n(\Z/N), \quad 1_n+NX \mapsto X \textrm{ mod } N\end{align*}
which is furthermore onto. Thus $\G(N)/\G(N^2)\simeq \mathfrak{sl}_n(\Z/N)\simeq (\Z/N, +)^{n^2-1}$, and the commutator subgroup $[\G(N),\G(N)]$ is contained in $\G(N^2)$. On the other hand, we have $1_n+N^2e_{ik}=[1_n+Ne_{ij},1_n+Ne_{jk}]\in [\G(N),\G(N)]$ for distinct $i,j, k$. At this point we use the fact that the principal congruence subgroup of level $M$ is normally generated by $\{1_n+Me_{ij}: i\neq j\}$, the $M$-th powers of the elementary matrices. This is what Mennicke actually proved in his approach to the Congruence Subgroup Property. As pointed out soon after by Bass - Milnor - Serre (Publ. Math. IHES 1967), this fact is equivalent to the Congruence Subgroup Property. It follows that $\G(N^2)$ is contained in $[\G(N),\G(N)]$, by the normality of $\G(N)$. Summarizing, $[\G(N),\G(N)]=\G(N^2)$, so that the abelianization of $\G (N)$ is $(\Z/N, +)^{n^2-1}$. Finally, if $\G(N)$ maps onto a non-trivial finite $p$-group then the abelianization of $\G(N)$ maps onto the corresponding abelianization, which is a non-trivial $p$-group, and we conclude that $p$ divides $N$.
\end{rem}

%%%%%%%%%%%%%%%%%%%%%%
\section{Proof of Platonov's theorem}\label{P}
Let $G$ be a finitely generated linear group over a field $K$, say $G\leq \GL_n(K)$. In $K$, consider the subring $A$ generated by the multiplicative identity $1$ and the matrix entries of a finite, symmetric set of generators for $G$. Thus $A$ is a finitely generated domain, and $G\leq \GL_n(A)$. Platonov's theorem is then a consequence of the following:

\begin{thm}\label{what we're actually proving}
Let $A$ be a finitely generated domain. If $\mathrm{char}\:A=0$, then $\GL_n(A)$ is virtually residually $p$-finite for all but finitely many primes $p$. If $\mathrm{char}\:A=p$, then $\GL_n(A)$ is virtually residually $p$-finite.
\end{thm}

The proof of Theorem~\ref{what we're actually proving} is a straightforward variation on the example of $\SL_n(\Z)$, as soon as we know the following facts:

\begin{lem}\label{technical lemma}
Let $A$ be a finitely generated domain. Then the following hold:
\begin{itemize}
\item[i.] $A$ is noetherian.
\item[ii.] $\cap_{k}\: I^k =0$ for any ideal $I\neq A$.
\item[iii.] If $A$ is a field, then $A$ is finite. 
\item[iv.] The intersection of maximal ideals of $A$ is $0$.
\item[v.] If $\mathrm{char}\:A=0$, then only finitely many primes $p=p\cdot 1$ are invertible in $A$.
\end{itemize}
\end{lem}

Let us postpone the proof of Lemma~\ref{technical lemma} for the moment, and see how to obtain Theorem~\ref{what we're actually proving}. The principal congruence subgroup of $\GL_n(A)$ corresponding to an ideal $I$ of $A$ is defined by
\begin{align*}\G(I)=\ker \big( \GL_n(A)\to \GL_n(A/I)\big).\end{align*}
If $\pi$ is a maximal ideal then $A/\pi$ is a finite field, by \ref{technical lemma} iii, so $\G(\pi)$ has finite index in $\GL_n(A)$. Also $\cap_{\pi}\: \G(\pi)=\{1_n\}$ as $\pi$ runs over the maximal ideals of $A$, by \ref{technical lemma} iv. This shows that $\GL_n(A)$ is residually finite, thereby proving Malcev's theorem.

We claim that $\pi^k/\pi^{k+1}$ is finite for each $k\geq 1$. In general, if $M$ is an $R$-module which is annihilated by an ideal $I$ -- in the sense that $IM=0$ -- then $M$ is also an $R/I$-module in a natural way: namely, define $\overline{r}\cdot m:=r\cdot m$ for $r\in R$ and $m\in M$. Furthermore, if $M$ is finitely generated as an $R$-module then $M$ is finitely generated as an $R/I$-module. In the case at hand, the $A$-module $\pi^k$ is finitely generated since $A$ is noetherian, so the $A$-module $\pi^k/\pi^{k+1}$ is also finitely generated. Therefore $\pi^k/\pi^{k+1}$ is finite dimensional as an $A/\pi$-module. As $A/\pi$ is finite, $\pi^k/\pi^{k+1}$ is finite as well.

The ring $A/\pi^k$ is finite for each $k\geq 1$, so each $\G(\pi^k)$ has finite index in $\GL_n(A)$. Furthermore, $\cap_{k}\: \G(\pi^k)=\{1_n\}$ by \ref{technical lemma} ii. (This shows once again that $\GL_n(A)$ is residually finite.) Let $p$ be the characteristic of $A/\pi$, so $p\in \pi$. Then $\G(\pi^{k})/\G(\pi^{k+1})$ is a $p$-group: for $X\in \G(\pi^k)$ we have
\begin{align*}X^p=1_n+\sum _{i=1}^p \binom{p}{i} (X-1_n)^i\in \G(\pi^{k+1}).\end{align*}
To conclude, $\GL_n(A)$ is virtually residually $p$-finite for each prime $p$ not invertible in $A$. By \ref{technical lemma} v, this happens for all but finitely many primes $p$ in the zero characteristic case. In characteristic $p$, there is only such prime, namely $p$ itself. Theorem~\ref{what we're actually proving} is proved.

We now return to the proof of the lemma. 

\begin{proof}[Proof of Lemma~\ref{technical lemma}] The first two points are standard: i) follows from the Hilbert Basis Theorem, and ii) is the Krull Intersection Theorem for domains.

iii) We use the following fact:
\begin{quotation}
Let $F\subseteq F(u)$ be a field extension with $F(u)$ finitely generated as a ring. Then $F\subseteq F(u)$ is a finite extension and $F$ is finitely generated as a ring.
\end{quotation}
Here is how we use this fact. Let $F$ be the prime field of $A$ and let $a_1,\dots,a_N$ be generators of $A$ as a ring. Thus $A=F(a_1,\dots,a_N)$. Going down the chain 
\begin{align*}
A=F(a_1,\dots,a_N)\supseteq F(a_1,\dots,a_{N-1})\supseteq \ldots\supseteq F
\end{align*} 
we obtain that $F\subseteq A$ is a finite extension, and that $F$ is finitely generated as a ring. Then $F$ is a finite field, as $\Q$ is not finitely generated as a ring, and so $A$ is finite.

Now here is how we prove the fact. Assume that $u$ is transcendental over $F$, i.e., $F(u)$ is the field of rational functions in $u$. Let $P_1/Q_1, \dots, P_N/Q_N$ generate $F(u)$ as a ring, where $P_i, Q_i\in F[u]$. The multiplicative inverse of $1+u\cdot \prod Q_i$ is a polynomial expression in the $P_i/Q_i$'s, which can be written as $R/\prod Q_i^{s_i}$. Therefore $\prod Q_i^{s_i}=(1+u\cdot \prod Q_i)R$ in $F[u]$. But this is impossible, since $\prod Q_i^{s_i}$ is relatively prime to $1+u\cdot \prod Q_i$. 

Thus $u$ is algebraic over $F$. Let $X^d+\alpha_1X^{d-1}+\dots+\alpha_d$ be the minimal polynomial of $u$ over $F$. Let also $a_1,\dots,a_N$ be ring generators of $F(u)=F[u]$. We may write each $a_i$ as $\sum_{0\leq m\leq d-1} \beta_{i,m} \: u^m$, with $\beta_{i,m} \in F$. We claim that the $\alpha_j$'s and the $\beta_{i,m}$'s are ring generators of $F$. Let $c\in F$. Then $c$ is a polynomial in $a_1,\dots,a_N$ over $F$, hence a polynomial in $u$ over the subring of $F$ generated by the $\beta_{i,m}$'s, hence a polynomial in $u$ of degree less than $d$ over the subring of $F$ generated by the $\alpha_j$'s and the $\beta_{i,m}$'s. By the linear independence of $\{1,u,\dots,u^{d-1}\}$, the latter polynomial is actually of degree $0$. Hence $c$ ends up in the subring of $F$ generated by the $\alpha_j$'s and the $\beta_{i,m}$'s.

iv) Let $a\neq 0$ in $A$. To find a maximal ideal of $A$ not containing $a$, we rely on the basic avoidance: maximal ideals do not contain invertible elements. Consider the localization $A'=A[1/a]$. Let $\pi'$ be a maximal ideal in $A'$, so $a\notin \pi'$. The restriction $\pi=\pi'\cap A$ is an ideal in $A$, and $a\notin \pi$. We show that $\pi$ is maximal. The embedding $A\hookrightarrow A'$ induces an embedding $A/\pi\hookrightarrow A'/\pi'$. As $A'/\pi'$ is a field which is finitely generated as a ring, in follows from iii) that $A'/\pi'$ is finite field. Therefore the subring $A/\pi$ is a finite domain, hence a field as well.

v) We shall use Noether's Normalization Theorem, which says the following. 
\begin{quotation}
Let $R$ be a finitely generated algebra over a field $F\subseteq R$. Then there are elements $x_1,\dots, x_N\in R$ algebraically independent over $F$ such that $R$ is integral over $F[x_1,\dots,x_N]$. 
\end{quotation}

In our case, $\Z$ is a subring of $A$, and $A$ is an integral domain which is finitely generated as a $\Z$-algebra. Extending to rational scalars, we have that $A_\Q=\Q\otimes_\Z A$ is a finitely generated $\Q$-algebra. By the Normalization Theorem, there exist elements $x_1,\dots, x_N$ in $A_\Q$ which are algebraically independent over $\Q$, and such that $A_\Q$ is integral over $\Q[x_1,\dots,x_N]$. Up to replacing each $x_i$ by an integral multiple of itself, we may assume that $x_1,\dots, x_N$ are in $A$. There is some positive $m\in \Z$ such that each ring generator of $A$  is integral over $\Z[1/m][x_1,\dots, x_N]$. Thus $A[1/m]$ is integral over the subring $\Z[1/m][x_1,\dots, x_N]$. If a prime $p$ is invertible in $A$, then it is also invertible in $A[1/m]$ while at the same time $p\in \Z[1/m][x_1,\dots, x_N]$. 

Now we use the following general fact. Let $R$ be a ring which is integral over a subring $S$. If $s\in S$ is invertible in $R$, then $s$ is already invertible in $S$. The proof is easy. Let $r\in R$ with $rs=1$. We have $r^d+s_{1}r^{d-1}+\dots+s_{d-1}r+s_d=0$ for some $s_i\in S$, since $r$ is integral over $S$. Multiplying through by $s^{d-1}$ yields $r\in S$.

Returning to our proof, we infer that $p$ is invertible in $\Z[1/m][x_1,\dots, x_N]$. By the algebraic independence of $x_1,\dots, x_N$, it follows that $p$ is actually invertible in $\Z[1/m]$. But only finitely many primes have this property, namely the prime factors of $m$.
\end{proof}

\begin{rem} Let $A$ be an infinite, finitely generated domain with $\mathrm{char}\:A=p>0$. 

If $n\geq 2$ then the $p$-torsion group $(A,+)$ embeds in $\GL_n(A)$, and this prevents $\GL_n(A)$ from being virtually residually $\ell$-finite for any prime $\ell\neq p$. So we cannot do any better in the positive characteristic case of Theorem~\ref{what we're actually proving}. 

Selberg's lemma fails in positive characteristic for a similar reason. The elementary group $\mathrm{E}_n(A)=\la 1_n+a\cdot e_{ij} : a\in A, i\neq j\ra$ is linear over the fraction field of $A$, and it fails to be virtually torsion-free since it contains copies of the infinite torsion group $(A,+)$. Furthermore, if $n\geq 3$ then $\mathrm{E}_n(A)$ is finitely generated. This is due to the commutator relations $[1_n+a\cdot e_{ij}, 1_n+b\cdot e_{jk}]=1_n+ab\cdot e_{ik}$ for distinct $i,j,k$, which imply that $\mathrm{E}_n(A)$ is generated by $\{1_n+a_1\cdot e_{ij},\dots, 1_n+a_N\cdot e_{ij} : i\neq j\}$ whenever $a_1, \dots, a_N$ are ring generators for $A$.  For a concrete example, take $A$ to be the polynomial ring $\mathbb{F}_p[t]$, in which case $\mathrm{E}_n(\mathbb{F}_p[t])=\SL_n(\mathbb{F}_p[t])$ since $\mathbb{F}_p[t]$ is a euclidean domain.
\end{rem}

\begin{rem}
Among finitely generated groups, we have the following implications:
\begin{align*}
\textrm{linear}\;\Rightarrow\;\textrm{virtually residually $p$-finite for some prime $p$}\;\Rightarrow\;\textrm{residually finite}
\end{align*}
The first implication, a ``$p$-adic'' refinement of Malcev's theorem, is an immediate consequence of Platonov's theorem. The second implication is Proposition~\ref{one p suffices}. Neither implication can be reversed, as witnessed by the following examples.

According to the previous remark, $\SL_n(\mathbb{F}_p[t])$ for $n\geq 3$ is finitely generated and virtually residually $\ell$-finite for $\ell=p$ only. Therefore $\SL_n(\mathbb{F}_{p}[t])\times \SL_n(\mathbb{F}_{q}[t])$, where $p$ and $q$ are different primes, is finitely generated, residually finite but not virtually residually $\ell$-finite for any prime $\ell$. 

The automorphism group of the free group on $n$ generators, $\mathrm{Aut} (F_{n})$, is virtually residually $p$-finite for all primes $p$. Indeed, as we have seen in Remark~\ref{what lurks}, free groups are residually $p$-finite for all primes $p$. Now a theorem of Lubotzky (J. Algebra 1980) says that $\mathrm{Aut} (G)$ is virtually residually $p$-finite whenever the finitely generated group $G$ is virtually residually $p$-finite. This is the ``$p$-adic'' analogue of an older, simpler, and better known theorem of G. Baumslag (J. London Math Soc. 1963) saying that $\mathrm{Aut} (G)$ is residually finite whenever the finitely generated group $G$ is residually finite. On the other hand, Formanek and Procesi (J. Algebra 1992) have shown that $\mathrm{Aut} (F_n)$ is not linear for $n\geq 3$. \end{rem}

%%%%%%%%%%%%%%%%%%%%%%

\end{document}